\documentclass{amsart}

\usepackage{amssymb}
\usepackage{amsfonts}
\usepackage{amsmath}
\usepackage{amsthm}
\usepackage{graphicx}
\usepackage{mathrsfs}
\usepackage{dsfont}
\usepackage{amscd}
\usepackage{multirow}
\usepackage{mathpazo}
\usepackage[all]{xy}
\usepackage[scaled]{helvet} % ss
\usepackage{courier} % tt
\normalfont
\usepackage[T1]{fontenc}
\usepackage{calligra}
\usepackage{verbatim}
\usepackage[usenames,dvipsnames]{color}
\usepackage[colorlinks=true,linkcolor=Blue,citecolor=Violet]{hyperref}

\newcommand{\N}{{\mathds{N}}}

\newcommand{\R}{{\mathds{R}}}

\newcommand{\D}{{\mathfrak{D}}}
\newcommand{\A}{{\mathfrak{A}}}
\newcommand{\B}{{\mathfrak{B}}}

\newcommand{\Lip}{{\mathsf{L}}}
\newcommand{\Hilbert}{{\mathscr{H}}}

\newcommand{\qpropinquity}[1]{{\mathsf{\Lambda}_{#1}}}

\newcommand{\Kantorovich}[1]{{\mathsf{mk}_{#1}}}

\newcommand{\Haus}[1]{{\mathsf{Haus}_{#1}}}

\newcommand{\StateSpace}{{\mathscr{S}}}

\newcommand{\mongekant}{{Mon\-ge-Kan\-to\-ro\-vich metric}}

\newcommand{\Lqcms}{{\JLL} quantum compact metric space}
\newcommand{\Qqcms}[1]{{$#1$}-quasi-Leibniz quantum compact metric space}
\newcommand{\gQqcms}{quasi-Leibniz quantum compact metric space}
\newcommand{\qcms}{quantum compact metric space}

\newcommand{\unit}{1}

\newcommand{\sa}[1]{{\mathfrak{sa}\left({#1}\right)}}

\newcommand{\inner}[2]{{\left<{#1},{#2}\right>}}
\newcommand{\JLL}{Lei\-bniz}

\newcommand{\dom}[1]{{\operatorname*{dom}({#1})}}
\newcommand{\diam}[2]{{\mathrm{diam}\left({#1},{#2}\right)}}

\newcommand{\opnorm}[2]{{\left|\mkern-1.5mu\left|\mkern-1.5mu\left| {#1} \right|\mkern-1.5mu\right|\mkern-1.5mu\right|_{#2}}}

\newcommand{\bridgelength}[2]{{\lambda\left({#1}\middle|{#2}\right)}}

\newcommand{\LipschitzD}{{\mathsf{LipD}}}

\newcommand{\alg}[1]{{\mathfrak{#1}}}
\newcommand{\norm}[2]{{\left\|{#1}\right\|_{#2}}}

\theoremstyle{plain}
\newtheorem{theorem}{Theorem}[section]

\newtheorem{claim}[theorem]{Claim}
\newtheorem{step}{Step}

\newtheorem{corollary}[theorem]{Corollary}

\newtheorem{lemma}[theorem]{Lemma}
\newtheorem{proposition}[theorem]{Proposition}

\newtheorem{theorem-definition}[theorem]{Theorem-Definition}

\theoremstyle{definition}
\newtheorem{definition}[theorem]{Definition}

\newtheorem{notation}[theorem]{Notation}

\newtheorem{convention}[theorem]{Convention}

\theoremstyle{remark}

\newtheorem{example}[theorem]{Example}

\newtheorem{remark}[theorem]{Remark}

\renewcommand{\geq}{\geqslant}
\renewcommand{\leq}{\leqslant}

\newcommand{\HausLip}[1]{{\mathrm{Haus}^\circ_{#1}}}
\numberwithin{equation}{section}

\begin{document}

\title{Equivalence of Quantum Metrics with a common domain}
\author{Fr\'{e}d\'{e}ric Latr\'{e}moli\`{e}re}
\thanks{This work is part of the project supported by the grant H2020-MSCA-RISE-2015-691246-QUANTUM DYNAMICS}
\email{frederic@math.du.edu}
\urladdr{http://www.math.du.edu/\symbol{126}frederic}
\address{Department of Mathematics \\ University of Denver \\ Denver CO 80208}

\date{\today}

\begin{abstract}
We characterize Lipschitz morphisms between quantum compact metric spaces as those *-morphisms which preserve the domain of certain noncommutative analogues of Lipschitz seminorms, namely lower semi-continuous Lip-norms. As a corollary, lower semi-continuous Lip-norms with a shared domain are in fact equivalent. We then note that when a family of lower semi-continuous Lip-norms are uniformly equivalent, then they give rise to totally bounded classes of quantum compact metric spaces, and we apply this observation to several examples of perturbations of quantum metric spaces. We also construct the noncommutative generalization of the Lipschitz distance between quantum compact metric spaces.
\end{abstract}
\maketitle

\tableofcontents

%%%%%%%%%%%%%%%%%%%%%%%%%%%%%%%%%%%%%%%%%%%

\newcommand{\Aut}[1]{{\mathrm{Aut}\left({#1}\right)}}

\section{Introduction}

Quantum metrics on C*-algebras, formally provided by generalized Lipschitz seminorms called \emph{Lip-norms} \cite{Rieffel98a,Rieffel99}, are the seeds for a new analytic framework which brings techniques from metric geometry in C*-algebra theory and provides a new tool set for problems from mathematical physics, such as finite dimensional approximations of quantum space-time \cite{Rieffel00,Rieffel01,Latremoliere05,Latremoliere13c,Rieffel10c,Rieffel15,Latremoliere16}, or perturbations of quantum metrics \cite{Latremoliere15b,Latremoliere15c}. Quantum compact metric spaces form a natural category \cite{Latremoliere15b}, whose morphisms are Lipschitz in an appropriate sense. In this paper, we prove that any *-morphism between two quantum compact metric spaces is actually Lipschitz if and only if it is compatible with the domains of the Lip-norms. In particular, *-automorphisms which preserve the domain of a Lip-norm must be bi-Lipschitz, and thus all Lip-norms with a common domain are actually equivalent. We then explore three related problems: we show that the topology of pointwise convergence on the group of Lipschitz automorphisms of a quantum compact metric space may be metrized using Lip-norms, and that our previous work on quantum perturbations naturally provide new examples of compact classes of quantum compact metric spaces for the quantum propinquity. We also construct the noncommutative generalization of the Lipschitz distance.

A compact quantum metric space is a generalization of a Lipschitz algebra, inspired by the work of Connes \cite{Connes89, Connes} and formalized by Rieffel \cite{Rieffel98a, Rieffel99}:

\begin{notation}
The space of self-adjoint elements in a C*-algebra $\A$ is denoted by $\sa{\A}$, while the state space of $\A$ is denoted by $\StateSpace(\A)$. The unit of a unital C*-algebra $\A$ is denoted by $\unit_\A$. Last, we denote the norm on a normed vector space $E$ by $\|\cdot\|_E$ by default.

Last, the diameter of a metric space $(E,\mathrm{d})$ is denoted by $\diam{E}{\mathrm{d}}$.
\end{notation}

\begin{convention}\label{conv-1}
  We adopt a convenient convention, when working with seminorms, throughout this paper. We use this convention accross our work: if $L$ is a seminorm defined on a subspace $\dom{L}$ of a vector space $E$, then we set $L(x) = \infty$ for all $x\notin\dom{L}$. Thus, with our convention, $\dom{L} = \{ x \in E : L(x) < \infty \}$.
\end{convention}

\begin{definition}[\cite{Rieffel98a, Rieffel99}]\label{qcms-def}
A pair $(\A,\Lip)$ is a \emph{quantum compact metric space} when $\A$ is a unital C*-algebra and $\Lip$ is a seminorm defined on a dense subspace $\dom{\Lip}$ of $\sa{\A}$, such that:
\begin{enumerate}
\item $\{a\in\dom{\Lip} : \Lip(a) = 0\} = \R\unit_\A$,
\item the \emph{\mongekant} $\Kantorovich{\Lip}$ defined for any two $\varphi,\psi \in \StateSpace(\A)$ by:
\begin{equation*}
\Kantorovich{\Lip}(\varphi,\psi) = \sup\left\{ |\varphi(a) - \psi(a)| : a\in\dom{\Lip}, \Lip(a)\leq 1 \right\}\text{,}
\end{equation*}
metrizes the weak* topology on $\StateSpace(\A)$.
\end{enumerate} 
If $(\A,\Lip)$ is a quantum compact metric space, then $\Lip$ is called a \emph{Lip-norm} on $\A$.
\end{definition}

The classical picture behind Definition (\ref{qcms-def}) is provided by a pair $(C(X),\mathsf{Lip})$ of the C*-algebra of a compact metric space $(X,\mathsf{d})$ and the Lipschitz seminorm $\mathsf{Lip}$ associated to the distance function $\mathsf{d}$. A generalization of quantum compact metric spaces to the quantum locally compact setting was proposed in \cite{Latremoliere05b,Latremoliere12b}. 

We may define a category whose objects are quantum compact metric spaces, and whose morphisms are a special type of *-morphisms between the underlying C*-algebras. There are at least two natural ideas. We may require that a Lipschitz morphism be a *-morphism which is also continuous with respect to the Lip-norms. Formally, if $(\A,\Lip_\A)$  and $(\B,\Lip_\B)$ are two quantum compact metric spaces, and $\varphi : \A\rightarrow\B$ is a *-morphism, this first approach to Lipschitz morphism consists in requiring that there exists $C > 0$ such that $\Lip_\B\circ\varphi \leq C\Lip_\A$. This relation imposes that $\varphi$ must be unital or null. Indeed, $\Lip_\B\circ\varphi(\unit_\A) \leq C \Lip_\A(\unit_\A) = 0$ so $\varphi(\unit_\A)\in\R\unit_\B$; since $\varphi$ is a *-morphism, this leaves us with $\varphi(\unit_\A) \in \{ 0, \unit_\B \}$. In this paper, we will work with unital *-morphisms.

Alternatively, we may require that the dual map associated to a unital *-morphism be a Lipschitz map between the state spaces equipped with their respective {\mongekant}. Continuing with our notations, we would thus ask that there exists $C > 0$ such that for all $\mu, \nu \in \StateSpace(\B)$, we have:
\begin{equation*}
\Kantorovich{\Lip_\A}(\mu\circ\varphi, \nu\circ\varphi) \leq C \Kantorovich{\Lip_\B}(\mu, \nu)\text{.}
\end{equation*}
In general, these two notions of a Lipschitz morphisms are not equivalent, owing to the fact that the {\mongekant} does not allow the recovery of the Lip-norm from which it was defined. After all, many Lip-norms may give the same {\mongekant}. 

However, among all Lip-norms which provide a given {\mongekant}, there is a particular one: the largest among all of them, which is characterized as being lower semi-continuous with respect to the norm of the underlying C*-algebra. In \cite{Rieffel99}, the study of this problem led to the notion of a closed Lip-norm, though the context there was more general (the underlying space was not a C*-algebra but a more general object called an order-unit space, which may not be complete, leading to some important subtleties).

For our purpose, it is thus natural to work with lower semi-continuous Lip-norms. In this context, our two notions of Lipschitz morphisms coincide. So we summarize our notion by:
\begin{definition}[{\cite{Latremoliere15b}}]
Let $(\A,\Lip_\A)$ and $(\B,\Lip_\B)$ be two quantum compact metric spaces, with $\Lip_\A$ and $\Lip_\B$ lower semi-continuous with respect to the norms of, respectively, $\A$ and $\B$. A unital *-morphism $\varphi : \A \rightarrow \B$ is \emph{$k$-Lipschitz} for some $k \geq 0$ when $\Lip_\B\circ\varphi \leq k \Lip_\A$, or equivalently:
\begin{equation*}
\mu \in \StateSpace(\B) \longmapsto \mu\circ\varphi \in \StateSpace(\A)
\end{equation*}
is a $k$-Lipschitz map from $(\StateSpace(\B),\Kantorovich{\Lip_\B})$ to $(\StateSpace(\A),\Kantorovich{\Lip_\A})$.
\end{definition}

It is easy to check that indeed, composition of Lipschitz morphisms is again Lipschitz, and the identity morphism is $1$-Lipschitz, so we have indeed defined a category. It is also easy to check that a $k$-Lipschitz morphism between two classical compact metric spaces is indeed of the form $f \in C(Y)\mapsto f\circ \theta$ with $\theta : X\rightarrow Y$ is $k$-Lipschitz.

In this paper, we investigate a third approach of Lipschitz morphisms. If $\dom{\Lip_\A}$ and $\dom{\Lip_\B}$ are the domains of $\Lip_\A$ and $\Lip_\B$, then a *-morphism may satisfy $\varphi(\dom{\Lip_\A}) \subseteq \dom{\Lip_\B}$. This appear to be a weaker notion, but as we shall see, it is again equivalent to the notion of a Lipschitz morphism. This reinforces that our notion of a category of quantum compact metric spaces is indeed appropriate.

Our paper then continues with an observation regarding sets of uniformly equivalent Lip-norms. We note that such sets are naturally compact for the quantum propinquity. The quantum Gromov-Hausdorff propinquity \cite{Latremoliere13} is a special member of the family of Gromov-Hausdorff propinquities \cite{Latremoliere13b,Latremoliere14}, which are all noncommutative analogues of the Gromov-Hausdorff distance \cite{Gromov, Gromov81} extending the topology of the latter to quantum compact metric spaces. The Gromov-Hausdorff propinquity provides a framework for the geometric study of classes of quantum compact metric spaces. We have established, for instance, the continuity of various natural families of C*-algebras such as quantum tori \cite{Latremoliere13c} or certain AF-algebras \cite{Latremoliere15d}. We constructed finite dimensional approximations for these spaces as well, answering some informal statements in mathematical physics. Another recent advance was the generalization of Gromov's theorem to our propinquity \cite{Latremoliere15}, providing us with an insight into the topological properties of various sets of compact quantum metric spaces.

In general, proving that a class of quantum compact metric spaces is indeed totally bounded for the propinquity may be subtle. For instance, \cite{Latremoliere15} relies on finite dimensional approximations, which are themselves a challenge. Nonetheless, as seen in \cite[Theorem 6.3]{Latremoliere15d}, our generalized Gromov's theorem can be put to use. Another approach taken in \cite{Latremoliere15d} was the construction of continuous maps from some compact spaces onto classes of quantum compact metric spaces; this method also applies to quantum tori and their finite dimensional approximations \cite{Latremoliere05,Latremoliere13c} and we shall see that it applies to conformal deformations as well in this paper. In this paper, we take yet another road to establish the compactness of some interesting classes of quantum compact metric spaces, obtained as perturbations of given quantum metrics.

Prior to the introduction of noncommutative Gromov-Hausdorff distances \cite{Rieffel00,kerr02,Latremoliere13,Latremoliere13b,Latremoliere14}, the idea of perturbations for a quantum metrics seemed to largely rely on the informal idea that certain algebraic expressions are qualitatively close to some original metric. We recently formalized this idea by actually establishing bounds on how far, in the sense of the propinquity, certain particular perturbations actually are. Examples of such perturbations include conformal deformations \cite{Ponge14} of quantum metrics arising from certain spectral triples \cite{Latremoliere15b}, leading to twisted spectral triples introduced by Connes and Moscovici \cite{Connes08} . We also brought curved quantum tori of Sitarz and D\k{a}browski \cite{Sitarz13,Sitarz15} into our program \cite{Latremoliere15c}. We shall see that a core ingredient of the constructions of these perturbations provide a uniform equivalence between Lip-norms, which in turn gives a compactness result.

We note that, besides the Gromov-Hausdorff distance, a standard extended metric between compact metric spaces is the Lipschitz distance. We provide in this paper a construction for the noncommutative version of the Lipschitz distance, which fits very well with the picture of Lipschitz morphisms presented in this paper.

The last section of this paper presents a natural metric on Lipschitz morphisms, built from quantum metrics.

\section{Equivalence of Lip-norms and Lipschitz morphisms}

Quantum metrics defined on the same domain are, in fact, equivalent, under a natural technical condition, which may always be assumed to no cost to the underlying metric structure. This observation is the subject of the following theorem, and constitutes our main result.

\begin{convention}
  Let $L$ be a seminorm defined on a dense subspace $\dom{L}$ of a normed vector space $E$. By Convention (\ref{conv-1}), we regard $L$ as a $[0,\infty]$-function over $E$. Now, we say that the seminorm $L$ is \emph{lower semicontinuous} on $E$ when $\left\{ x \in \dom{L} : L(x) \leq 1 \right\}$ is closed in $E$. Note that this is \emph{stronger} than requiring $L$ to be lower semicontinuous as a function from $dom{L}$.
\end{convention}

\begin{theorem}\label{equiv-thm}
  Let $(\A,\Lip)$ be a quantum compact metric space, where the domain of $\Lip$ is denoted by $\dom{\Lip}$, and where $\Lip$ is lower semi-continuous on $\sa{\A}$, i.e.
  \begin{equation*}
    \left\{ a\in\dom{\Lip} : \Lip(a) \leq 1 \right\}
  \end{equation*}
  is closed in $\A$.
  
  Let $\mathsf{S}$ be a seminorm on $\dom{\Lip}$ such that:
\begin{enumerate}
\item $\mathsf{S}$ is lower semi-continuous on $\sa{\A}$ with respect to $\|\cdot\|_\A$, i.e. $\{ a \in \dom{\mathsf{S}} : \mathsf{S}(a)\leq 1\}$ is a closed subset of $(\A,\norm{\cdot}{\A})$,
\item $\mathsf{S}(\unit_\A) = 0$.
\end{enumerate}
Then there exists $C > 0$ such that for all $a\in\dom{\Lip}$:
\begin{equation*}
\mathsf{S}(a) \leq C \Lip(a)\text{.}
\end{equation*}
\end{theorem}

\begin{proof}
Our proof proceeds in three steps. First, we prove that the domain of a lower semi-continuous seminorm can be made naturally into a Banach space. Then, we use the open mapping theorem to show that different lower semi-continuous seminorms defined on the same domain give rise to equivalent Banach norms with the construction in our first step. Last, we conclude our theorem.

\begin{step}
  Let $\mathsf{S}$ be a seminorm defined on some dense subspace $\dom{\mathsf{S}}$ of $\sa{\A}$, such that $\{a\in\dom{\mathsf{S}}:\mathsf{S}(a)\leq 1\}$ is closed in $\sa{\A}$. Let $\|\cdot\|_{\mathsf{S}} = \|\cdot\|_\A + {\mathsf{S}}$. We first check that $\dom{{\mathsf{S}}}$ is a Banach space for the norm $\|\cdot\|_{\mathsf{S}}$. 
\end{step}

It is straightforward that $\|\cdot\|_{\mathsf{S}}$ is a norm on $\dom{{\mathsf{S}}}$.

Let $(a_n)_{n\in\N}$ be a Cauchy sequence in $\dom{{\mathsf{S}}}$ for $\|\cdot\|_{\mathsf{S}}$. Thus $(a_n)_{n\in\N}$ is a Cauchy sequence for $\|\cdot\|_\A$, which is complete, so $(a_n)_{n\in\N}$ converges to some $a \in \sa{\A}$ for $\|\cdot\|_\A$. 

We also observe that $(a_n)_{n\in\N}$ is Cauchy, hence bounded, for ${\mathsf{S}}$; thus there exists $M > 0$ such that ${\mathsf{S}}(a_n)\leq M$ for all $n\in\N$. As $\{ a \in \dom{\mathsf{S}} : \mathsf{S}(a)\leq M \}$ is closed in $\A$, we conclude that $a \in \dom{\mathsf{S}}$ and $\mathsf{S}(a)\leq M$ (alternatively, with our convention that $\mathsf{S}(a) \in [0,\infty)$, our assumption on the unit ball of $\mathsf{S}$ is equivalent to requiring $\mathsf{S}$ to be lower semicontinuous on $\A$ as a $[0,\infty]$-valued function, and thus we obtain that ${\mathsf{S}}(a) \leq \liminf_{n\rightarrow\infty} {\mathsf{S}}(a_n) \leq M$; thus $a\in \dom{{\mathsf{S}}}$).

Let $\varepsilon > 0$. Since $(a_n)_{n\in\N}$ is Cauchy for ${\mathsf{S}}$, there exists $N\in\N$ such that for all $p,q\geq N$ we have ${\mathsf{S}}(a_p-a_q) \leq \varepsilon$. Since ${\mathsf{S}}$ is lower semi-continuous with respect to $\|\cdot\|_\A$, we thus have, for all $p\geq N$:
\begin{equation*}
{\mathsf{S}}(a-a_p)\leq \liminf_{q\rightarrow\infty} {\mathsf{S}}(a_p-a_q) \leq \varepsilon\text{.}
\end{equation*}
Thus $\lim_{p\rightarrow\infty} {\mathsf{S}}(a-a_p) = 0$. This proves that $\lim_{n\rightarrow\infty} \|a-a_n\|_{\mathsf{S}} = 0$ as desired.

\begin{step}
If $\Lip$ and $\mathrm{S}$ be two lower semi-continuous seminorms on some common dense subspace $\dom{\Lip}$ of $\sa{\A}$, then the norms:
\begin{equation*}
\|\cdot\|_{\Lip} = \|\cdot\|_\A + \Lip \text{ and }\|\cdot\|_{\mathsf{S}} = \|\cdot\|_\A + \mathsf{S}
\end{equation*}
are equivalent.
\end{step}

The norms $\|\cdot\|_{\Lip}$  and $\|\cdot\|_{\mathsf{S}}$ both make $\dom{\Lip}$ into a Banach space, by step 1.

We begin with a simple observation. Let $(x_n)_{n\in\N}$ be a sequence in $\dom{\Lip}$ which converges for both $\|\cdot\|_{\Lip}$ and $\|\cdot\|_{\mathsf{S}}$. Let $x \in \dom{\Lip}$ be the limit of $(x_n)_{n\in\N}$ for $\|\cdot\|_{\Lip}$ and $y \in \dom{\Lip}$ be the limit of $(x_n)_{n\in\N}$ for $\|\cdot\|_{\mathsf{S}}$. We note that in particular, $(x_n)_{n\in\N}$ converges to both $x$ and $y$ for $\|\cdot\|_\A$. Thus $x=y$.

Let now $\|\cdot\|_\ast = \|\cdot\|_{\Lip} + \|\cdot\|_{\mathsf{S}}$. If $(x_n)_{n\in\N}$ is a Cauchy sequence for $\|\cdot\|_\ast$, then the sequence $(x_n)_{n\in\N}$ is Cauchy for both $\|\cdot\|_{\Lip}$ and $\|\cdot\|_{\mathsf{S}}$ and thus converges for both these norms, since they are complete; by our previous observation, $(x_n)_{n\in\N}$ has the same limit $x\in\dom{\Lip}$ for both these norms. Hence, $(x_n)_{n\in\N}$ converges to $x$ for $\|\cdot\|_\ast$, i.e. $\left(\dom{\Lip},\|\cdot\|_\ast\right)$ is a Banach space.

Now, since $\|\cdot\|_{\Lip}\leq \|\cdot\|_\ast$, the open mapping theorem \cite{Brezis83} implies that there exists $k > 0$ such that $\|\cdot\|_\ast\leq k\|\cdot\|_{\Lip}$. We then conclude easily that $\|\cdot\|_{\mathsf{S}}\leq k \|\cdot\|_{\Lip}$. Similarly, for some $k' > 0$, we have $\|\cdot\|_{\Lip}\leq k' \|\cdot\|_{\mathsf{S}}$.

\begin{step}
We now conclude our theorem.
\end{step}
We thus are given a lower semi-continuous Lip-norm $\Lip$ on $\A$, and some seminorm $\mathrm{S}$ on the domain of $\Lip$, with $\mathrm{S}(\unit_\A) = 0$ and $\mathrm{S}$ lower semi-continuous with respect to $\|\cdot\|_\A$.

Using our previous step, there exists $k > 0$ such that for all $a\in\dom{\Lip}$, we have:
\begin{equation*}
\begin{split}
\mathsf{S}(a) &= \|a\|_{\mathsf{S}} - \|a\|_\A \\
&\leq k\|a\|_{\Lip} - \|a\|_\A \\
&\leq k\left(\|a\|_\A + \Lip(a)\right) - \|a\|_\A \\
&= (k-1)\|a\|_\A + k\Lip(a)\text{.}
\end{split}
\end{equation*}
Let $D = \diam{\StateSpace(\A)}{\Kantorovich{\Lip}}$ (which is finite since $\Lip$ is a Lip-norm). Let $a\in\dom{\Lip_\A}$. As $\Lip$ is a Lip-norm, there exists $t\in\R$ such that:
\begin{equation}\label{equiv-thm-exp1}
\|a+t\unit_\A\|_\A \leq D \Lip(a)\text{.}
\end{equation}
Indeed, let $\mu\in\StateSpace(\A)$. For all $\nu\in\StateSpace(\A)$ then:
\begin{equation*}
|\mu(a)-\nu(a)|\leq \Lip(a)\Kantorovich{\Lip}(\mu,\nu) \leq \Lip(a)D\text{.}
\end{equation*}

Of course, $\Lip(a+t\unit_\A) = \Lip(a)$ and $\mathsf{S}(a+t\unit_\A) = \mathsf{S}(a)$ since $\mathsf{S}(\unit_\A) = 0$. Thus:
\begin{equation*}
\begin{split}
\mathsf{S}(a) &= \mathsf{S}(a + t\unit_\A) \\
&\leq (k-1)\|a + t\unit_\A\|_\A + k \Lip(a + t\unit_\A) \\
&\leq (k-1) D \Lip(a) + k\Lip(a) = ((k-1)D + k) \Lip(a)\text{.}
\end{split} 
\end{equation*}

This concludes our proof, with $C = ((k-1)D+k)$.
\end{proof}

Theorem (\ref{equiv-thm}) has the following consequences:

\begin{corollary}\label{Lipschitz-cor}
Let $(\A,\Lip_\A)$ and $(\B,\Lip_\B)$ be two compact quantum metric spaces whose Lip-norms are lower semi-continuous on, respectively, $\sa{\A}$ and $\sa{\B}$. Let $\varphi : \A\rightarrow\B$ be some unital *-morphism such that $\varphi(\dom{\Lip_\A}) \subseteq \dom{\Lip_\B}$. Then there exists $C > 0$ such that:
\begin{equation*}
\Lip_\B\circ\varphi \leq C \Lip_\A
\end{equation*}
i.e. $\varphi$ is a $C$-Lipschitz morphism.
\end{corollary}

\begin{proof}
Let $\mathsf{S} = \Lip_\B\circ\varphi$. We note that $\mathsf{S}$ is a lower semi-continuous seminorm which takes finite values on $\dom{\Lip_\A}$ by assumption. Moreover $\mathsf{S}(\unit_\A) = 0$. Thus our corollary follows from Theorem (\ref{equiv-thm}).
\end{proof}

\begin{corollary}\label{bi-Lipschitz-cor}
Let $(\A,\Lip_\A)$ and $(\B,\Lip_\B)$ be two quantum compact metric spaces whose Lip-norms are lower semi-continuous, respectively, on $\sa{\A}$ and $\sa{\B}$. If $\varphi : \A\rightarrow\B$ is a *-isomorphism such that $\varphi(\dom{\Lip_\A}) = \dom{\Lip_\B}$, then there exists $C > 0$ such that:
\begin{equation*}
C^{-1} \Lip_\A \leq \Lip_\B\circ\varphi \leq C \Lip_\A\text{.}
\end{equation*}
\end{corollary}

\begin{proof}
This follows from Corollary (\ref{Lipschitz-cor}).
\end{proof}

\begin{corollary}\label{auto-Lipschitz-cor}
  Let $(\A,\Lip)$ be a quantum compact metric space where $\Lip$ is lower semi-continuous over $\sa{\A}$ and $\sa{\B}$, respectivly. Let $\alpha$ be a *-automorphism of $\A$. The following two assertions are equivalent:
\begin{enumerate}
\item $\alpha(\dom{\Lip}) = \dom{\Lip}$,
\item there exists $C > 0$ such that $C^{-1} \Lip\circ\alpha \leq \Lip \leq C \Lip\circ\alpha$.
\end{enumerate}
\end{corollary}

\begin{proof}
Assume (2) first. Let $a\in\dom{\Lip}$. Then:
\begin{equation*}
\Lip(\alpha^{-1}(a)) \leq C \Lip(\alpha(\alpha^{-1}(a))) = \Lip(a) < \infty\text{,}
\end{equation*}
hence $\alpha^{-1}(a) \in \dom{\Lip}$. Thus $a = \alpha(\alpha^{-1}(a)) \in \alpha(\dom{\Lip})$, i.e. $\dom{\Lip} \subseteq \alpha(\dom{\Lip})$. The converse inclusion is proven similarly.

Assume (1). Then (2) follows from Corollary (\ref{bi-Lipschitz-cor}).
\end{proof}

\begin{corollary}\label{equiv-cor}
Let $\A$ be a unital C*-algebra and $\Lip_1$, $\Lip_2$ be two lower semi-continuous Lip-norms on $\A$. The following assertions are equivalent:
\begin{enumerate}
\item $\Lip_1$ and $\Lip_2$ are equivalent,
\item $\dom{\Lip_1} = \dom{\Lip_2}$.
\end{enumerate}
\end{corollary}

\begin{proof}
Apply Corollary (\ref{auto-Lipschitz-cor}) to the identity automorphism of $\A$.
\end{proof}

\section{Compactness of classes of perturbations of {\gQqcms s}}

To present the result in this section, we introduce a simple metric on Lip-norms over a fixed C*-algebra. We will adopt the following terminology, to keep our statements readable.
\begin{convention}
  If $(\A,\Lip)$ is a quantum compact metric space, then we say that $\Lip$ is llower semicontinuous to mean that $\Lip$ is lower semi-continuous, as a $[0,\infty]$-valued function, over $\sa{\A}$.
\end{convention}

\begin{definition}\label{HausLip-def}
For any two lower semi-continuous Lip-norms $\Lip$, $\Lip'$ on a unital C*-algebra $\A$, we define $\HausLip{\A}(\Lip,\Lip')$ by:
\begin{equation*}
\HausLip{\A}(\Lip,\Lip') = \Haus{\|\cdot\|_\A}\left(\left\{a\in \sa{\A} : \Lip(a)\leq 1\right\}, \{a\in \sa{\A}: \Lip'(a)\leq  1\} \right)\text{,}
\end{equation*}
where $\Haus{\|\cdot\|_\A}$ is the Hausdorff distance induced on closed subsets of $\A$ by the norm $\|\cdot\|_\A$.
\end{definition}

We begin by observing that our distance $\HausLip{\A}$ is indeed finite. Let us start by recalling the following fundamental characterization of quantum compact metric spaces proved by Rieffel in \cite{Rieffel98a}, and akin to a noncommutative Arz{\'e}la-Ascoli theorem:

\begin{theorem}[{\cite{Rieffel98a, Rieffel99}}]\label{Rieffel-az-thm}
Let $\A$ be a unital C*-algebra and $\Lip$ a seminorm defined on a dense subspace $\dom{\Lip}$ of $\sa{\A}$ and such that $\{a\in\dom{\Lip}:\Lip(a) = 0\} = \R\unit_\A$. The following assertions are equivalent:
\begin{enumerate}
\item $(\A,\Lip)$ is a quantum compact metric space,
\item $\{a\in\dom{\Lip} : \Lip(a) \leq 1, \varphi(a) = 0\}$ is totally bounded in norm for some state $\varphi\in\StateSpace(\A)$,
\item $\{a\in\dom{\Lip} : \Lip(a) \leq 1, \varphi(a) = 0\}$ is totally bounded in norm for all states $\varphi\in\StateSpace(\A)$,
\item $\{a\in\dom{\Lip} : \Lip(a)\leq 1, \|a\|_\A\leq 1\}$ is totally bounded in norm.
\end{enumerate}
In particular, if $\Lip$ is lower semi-continuous on $\sa{\A}$, then $\Lip$ is a Lip-norm if, and only if any of the sets above are compact.
\end{theorem}

\begin{lemma}\label{finite-HausLip-lemma}
Let $\A$ be a unital C*-algebra. For any two lower semi-continuous Lip-norms $\Lip_1$ and $\Lip_2$ on $\A$, and for any state $\varphi \in \StateSpace(\A)$, we have:
\begin{multline*}
\HausLip{\A}(\Lip_1,\Lip_2) \leq \Haus{\|\cdot\|_\A}(\{a\in\sa{\A}:\varphi(a) = 0, \Lip_1(a) \leq 1\}, \\ \{a\in\sa{\A}:\varphi(a)=0, \Lip_2(a)\leq 1\}) \text{.}
\end{multline*}
In particular, $\HausLip{\A}(\Lip_1,\Lip_2)$ is finite.
\end{lemma}

\begin{proof}
Let $\varphi \in \StateSpace(\A)$. By Theorem (\ref{Rieffel-az-thm}), $\{a\in\sa{\A} : \varphi(a) = 0, \Lip_j(a)\leq 1\}$ are compact for $j=1,2$. Thus the Hausdorff distance (for the norm of $\A$) between these two sets is finite; let us denote it by $d$.

Let $a\in\sa{\A}$ with $\Lip_1(a)\leq 1$. There exists $b \in \sa{\A}$ with $\Lip_2(b) \leq 1$ and $\varphi(b) = 0$ such that $\|a-\varphi(a)\unit_\A - b\|_\A \leq d$. Thus $\|a - (b + \varphi(a)\unit_\A)\|_\A \leq d$, and we note that $\Lip_2(b + \varphi(a)\unit_\A) = \Lip_2(b) \leq 1$. As the argument is symmetric is $\Lip_1$ and $\Lip_2$, we have shown our lemma.
\end{proof}

Our goal is to establish a new sufficient condition for certain classes of quantum compact metric spaces to be totally bounded for the \emph{quantum propinquity}. We refer to \cite{Latremoliere13,Latremoliere13b,Latremoliere14,Latremoliere15,Latremoliere15b} for the definition of the quantum propinquity and some of its properties. We briefly recall from \cite{Latremoliere13} the notion of a bridge, as it will be used in our next proof, and provide a characterization of the quantum propinquity.

Let $\A$ and $\B$ be two unital C*-algebras. A \emph{bridge} $(\D,\omega,\pi_\A,\pi_\B)$ from $\A$ to $\B$ is a unital C*-algebra $\D$ and two unital *-monomorphisms $\pi_\A : \A \hookrightarrow\D$ and $\pi_\B:\B \hookrightarrow\D$, as well as an element $\omega\in\D$ such that for at least one state $\varphi$ of $\D$, we have $\varphi(\omega d) = \varphi(d \omega) = \varphi(d)$ for all $d\in\D$. The set of all such states of $\D$, denoted by $\StateSpace_1(\D|\omega)$, is the $1$-level set of $\D$.

Now if $(\A,\Lip_\A)$ and $(\B,\Lip_\B)$ are two quantum compact metric spaces, then we can associate a number, called the length, to a bridge $\gamma = (\D,\omega,\pi_\A,\pi_\B)$ from $\A$ to $\B$. We first define the \emph{reach} of $\gamma$ as the Hausdorff distance between $\{ \pi_\A(a)\omega : a\in\sa{\A}, \Lip_\A(a)\leq 1\}$ and $\{ \omega\pi_\A(b) : b\in\sa{\B}, \Lip_\B(b)\leq 1\}$ for $\Haus{\|\cdot\|_\D}$. We then define the \emph{height} of $\gamma$ as the maximum of the Hausdorff distance, for $\Haus{\Kantorovich{\Lip_\A}}$, between $\StateSpace(\A)$ and $\{\varphi\circ\pi_\A : \varphi\in\StateSpace_1(\D|\omega)\}$, and the Hausdorff distance for $\Haus{\Kantorovich{\Lip_\B}}$, between $\StateSpace(\B)$ and $\{\varphi\circ\pi_\B : \varphi\in\StateSpace_1(\D|\omega)\}$.

The length $\bridgelength{\gamma}{\Lip_\A,\Lip_\B}$ of the bridge $\gamma$ is the maximum of its reach and its height. The quantum propinquity is constructed from bridges, although it requires a few technical steps. In particular, the quantum propinquity is defined on classes of {\Qqcms{F}} for an \emph{admissible function $F$}, i.e. a function $F : [0,\infty)^4\rightarrow[0,\infty)$ which is increasing for the product order on $[0,\infty)^4$ and such that $F(x,y,l_x,l_y) \geq xl_y + y l_x$ for all $x,y,l_x,l_y \geq 0$. Given such a function, a {\Qqcms{F}} $(\A,\Lip)$ is a quantum compact metric space such that for all $a,b \in \sa{\A}$ we have:
\begin{equation*}
\max\left\{ \Lip\left(\frac{a b + b a}{2}\right), \Lip\left(\frac{a b - b a}{2i} \right)  \right\} \leq F(\|a\|_\A,\|b\|_\A,\Lip(a),\Lip(b)) \text{.}
\end{equation*}

The following result characterizes the quantum propinquity.

\begin{theorem-definition}[\cite{Latremoliere13}]\label{def-thm}
Let $\mathcal{L}$ be the class of all {\Qqcms{F}s} for some admissible function $F$. There exists a class function $\qpropinquity{F}$ from $\mathcal{L}\times\mathcal{L}$ to $[0,\infty) \subseteq \R$ such that:
\begin{enumerate}
\item for any $(\A,\Lip_\A), (\B,\Lip_\B) \in \mathcal{L}$ we have:
\begin{equation*}
0\leq \qpropinquity{F}((\A,\Lip_\A),(\B,\Lip_\B)) \leq \max\left\{\diam{\StateSpace(\A)}{\Kantorovich{\Lip_\A}}, \diam{\StateSpace(\B)}{\Kantorovich{\Lip_\B}}\right\}\text{,}
\end{equation*}
\item for any $(\A,\Lip_\A), (\B,\Lip_\B) \in \mathcal{L}$ we have:
\begin{equation*}
\qpropinquity{F}((\A,\Lip_\A),(\B,\Lip_\B)) = \qpropinquity{F}((\B,\Lip_\B),(\A,\Lip_\A))\text{,}
\end{equation*}
\item for any $(\A,\Lip_\A), (\B,\Lip_\B), (\alg{C},\Lip_{\alg{C}}) \in \mathcal{L}$ we have:
\begin{equation*}
\qpropinquity{F}((\A,\Lip_\A),(\alg{C},\Lip_{\alg{C}})) \leq \qpropinquity{F}((\A,\Lip_\A),(\B,\Lip_\B)) + \qpropinquity{F}((\B,\Lip_\B),(\alg{C},\Lip_{\alg{C}}))\text{,}
\end{equation*}
\item for all $(\A,\Lip_\A), (\B,\Lip_\B) \in \mathcal{L}$ and for any bridge $\gamma$ from $\A$ to $\B$, we have:
\begin{equation*}
\qpropinquity{F}((\A,\Lip_\A), (\B,\Lip_\B)) \leq \bridgelength{\gamma}{\Lip_\A,\Lip_\B}\text{,}
\end{equation*}
\item for any $(\A,\Lip_\A), (\B,\Lip_\B) \in \mathcal{L}$, we have $\qpropinquity{F}((\A,\Lip_\A),(\B,\Lip_\B)) = 0$ if and only if $(\A,\Lip_\A)$ and $(\B,\Lip_\B)$ are isometrically isomorphic, i.e. if and only if there exists a *-isomorphism $\pi : \A \rightarrow\B$ with $\Lip_\B\circ\pi = \Lip_\A$, or equivalently there exists a *-isomorphism $\pi : \A \rightarrow\B$ whose dual map $\pi^\ast$ is an isometry from $(\StateSpace(\B),\Kantorovich{\Lip_\B})$ into $(\StateSpace(\A),\Kantorovich{\Lip_\A})$,

\item if $\Xi$ is a class function from $\mathcal{L}\times \mathcal{L}$ to $[0,\infty)$ which satisfies Properties (2), (3) and (4) above, then $\Xi((\A,\Lip_\A), (\B,\Lip_\B)) \leq \qpropinquity{F}((\A,\Lip_\A),(\B,\Lip_\B))$ for all $(\A,\Lip_\A)$ and $(\B,\Lip_\B)$ in $\mathcal{L}$,
\item the topology induced by $\qpropinquity{F}$ on the class of classical metric spaces agrees with the topology induced by the Gromov-Hausdorff distance.
\end{enumerate}
\end{theorem-definition}

We connect our new distance between Lip-norms on a fixed C*-algebra and the propinquity easily.

\begin{proposition}\label{HausLip-qprop-prop}
Let $\A$ be a unital C*-algebra. If $\Lip_1$ and $\Lip_2$ are two $F$-quasi-Leibniz lower semi-continuous Lip-norms on $\A$ for some admissible function $F$, then:
\begin{equation*}
\qpropinquity{F}((\A,\Lip_1),(\A,\Lip_2)) \leq \HausLip{\A}(\Lip_1,\Lip_2)\text{.}
\end{equation*}
\end{proposition}

\begin{proof}
We simply use the bridge $(\A,\unit_\A,\mathrm{id},\mathrm{id})$ where $\mathrm{id}$ is the identity on $\A$.
\end{proof}

The purpose of this section is to establish the fact that uniformly equivalent Lip-norms, as defined in the hypothesis of the next proposition, provide totally bounded classes of quantum compact metric spaces for the metric $\HausLip{}$ and thus for the quantum propinquity, whenever applicable.

\begin{proposition}\label{totally-bounded-prop}
Let $(\A,\Lip)$ be a quantum compact metric space where $\Lip$ is lower semi-continuous. If $\Xi$ is a set of lower semi-continuous Lip-norms on $\A$ for which there exists $C > 0$ such that, for all $\mathsf{Lip} \in \Xi$, we have $\Lip \leq C \mathsf{Lip}$, then $\Xi$ is totally bounded for $\HausLip{\A}$ (and therefore, when applicable, for the quantum propinquity as well).
\end{proposition}

\begin{proof}
We fix $\mu \in \StateSpace(\A)$. By assumption, for all $\mathsf{Lip} \in \Xi$, we have:
\begin{equation*}
\left\{ a\in\sa{\A} : \mathsf{Lip}(a) \leq 1, \mu(a) = 0 \right\} \subseteq \left\{ a\in\sa{\A} : \Lip(a)\leq C, \mu(a) = 0 \right\}\text{.}
\end{equation*}
Now, since $\Lip$ is a lower semi-continuous Lip-norm, the set:
\begin{equation*}
\alg{L} = \left\{ a\in\sa{\A} : \Lip(a)\leq C, \mu(a) = 0 \right\}
\end{equation*}
is compact for $\|\cdot\|_\A$ by Theorem (\ref{Rieffel-az-thm}). Thus by Blaschke's Theorem, the hyperspace of the closed subsets of $\alg{L}$ is compact for the Hausdorff distance $\Haus{\|\cdot\|_\A}$. We conclude our proof using Lemma (\ref{finite-HausLip-lemma}) (and Proposition (\ref{HausLip-qprop-prop}) for the quantum propinquity conclusion).
\end{proof}

The main application of Proposition (\ref{totally-bounded-prop}) in this paper concerns certain perturbations we have established in \cite{Latremoliere15,Latremoliere15b}. These perturbations were constructed using \cite[Lemma 3.79]{Latremoliere15b}, which encapsulates a recurrent computation when estimating the quantum propinquity between two {\gQqcms s}. When using this lemma, one will typically obtain uniformly equivalent families of Lip-norms, thus the following examples will be typical.

\begin{notation}
If $\Hilbert$ is some Hilbert space and $T$ is a bounded linear operator on $\Hilbert$, then we denote the norm of $T$ by $\opnorm{T}{\Hilbert}$.
\end{notation}

\begin{proposition}[{\cite[Proposition 3.82]{Latremoliere15b}}]
  Let $\A$ be a unital C*-algebra, $\pi$ a unital faithful *-representation of $\A$ on some Hilbert space $\Hilbert$, and $D$ a self-adjoint, possibly unbounded operator on $\Hilbert$ such that setting:
  \begin{equation*}
    \dom{\Lip} = \left\{ a \in \sa{\A} : \pi(a)\dom{D}\subseteq\dom{D} \text{ and }[D,\pi(a)] \text{ is bounded } \right\}
  \end{equation*}
  and
  \begin{equation*}
    \Lip : a\in\A \longmapsto \opnorm{[D,\pi(a)]}{\Hilbert}
  \end{equation*}
  the pair $(\A,\Lip)$ is a {\Lqcms}. Let $r = \diam{\A}{\Lip}$ and assume $r > 0$.
  
  Let $\B$ be the C*-algebra of all bounded linear operators on $\Hilbert$. 
  
  For any $\omega \in \sa{\B}$ with $\norm{\omega}{\B} < \frac{1}{2r}$, we define:
  \begin{equation*}
    D_\omega = D+\omega\text{ and }\Lip_\omega : a\in\dom{\Lip} \mapsto \opnorm{[D_\omega, \pi(a)]}{\Hilbert}\text{.}
  \end{equation*}
  
  The pair $(\A,\Lip_\omega)$ is a {\Lqcms} for all bounded self-adjoint $\omega$ on $\Hilbert$ \emph{such that $\norm{\omega}{\B}<\frac{1}{2r}$}, and, moreover:
  \begin{equation*}
    \omega \in \{\xi\in\sa{\B}:\norm{\xi}{\B} \} \longmapsto (\A,\Lip_\omega)
  \end{equation*}
  is continuous for the quantum Gromov-Hausdorff propinquity $\qpropinquity{}$.
\end{proposition}

\begin{proof}
By \cite[Proposition 3.7]{Rieffel00}, the seminorm $\Lip$ is lower semi-continuous.
  
For all $a\in\dom{\Lip}$ and for all $t\in\R$, we have
\begin{align*}
  \left|\Lip(a)-\Lip_\omega(a)\right|
  & \leq \norm{[-\omega,a]}{\B} = \norm{[\omega, a-t\unit_\B]}{\B}  \\
  & \leq 2 \norm{\omega}{\B} \norm{a-t\unit_\A}{\A}\text.
\end{align*}
Since $\Lip$ is a Lip-norm, there exists $t\in\R$ such that $\norm{a-t\unit_\A}{\A} \leq r \Lip(a)$. Thus we conclude
\begin{equation*}
  \forall a \in \dom{\Lip} \quad \left|\Lip(a)-\Lip_\omega(a)\right| \leq 2 r \norm{\omega}{\B} \Lip(a) \text.
\end{equation*}
Thus
\begin{equation*}
  \forall a \in \dom{\Lip} \quad (1-2 r \norm{\omega}{\B}) \Lip(a) \leq \Lip_\omega(a) \leq (1 + 2 r \norm{\omega}{\B}) \Lip(a) \text.
\end{equation*}

Since $\norm{\omega}{\B}<\frac{1}{2r}$, we conclude from \cite[Lemma 1.10]{Rieffel98a} that, indeed, $(\A,\Lip_\omega)$ is a {\qcms}. Thus we see that if $\Omega = \{ \omega \in \sa{\B} : \|\omega\|_\B \leq C\}$ for some $C \in (0,\frac{1}{2r})$, then the set of quantum metrics:
\begin{equation*}
\{\Lip_\omega : \omega \in \Omega \}
\end{equation*}
is totally bounded by Proposition (\ref{totally-bounded-prop}) for both $\HausLip{\A}$ and $\qpropinquity{}$. Yet $\Omega$ is not totally bounded unless $\B$ is finite dimensional.
\end{proof}

A deeper example of a non trivial class of totally bounded quantum metrics obtained from perturbations and \cite[Lemma 3.79]{Latremoliere15} is given by curved quantum tori, where, once more, the space of parameters is not totally bounded. These examples come from mathematical physics and were treated in \cite{Latremoliere15c} from the metric perspective.

\begin{example}[Curved quantum tori]
\begin{theorem}[{\cite[Theorem III.1]{Latremoliere15c}}]\label{ctori-thm}
Let $\A$ be a unital C*-algebra, and let $\alpha$ be a strongly continuous ergodic action of a compact Lie group $G$ on $\A$. Let $n$ be the dimension of $G$. We endow the dual $\alg{g}'$ of the Lie algebra $\alg{g}$ of $G$ with an inner product $\inner{\cdot}{\cdot}$, and we denote by $C$ the Clifford algebra of $(\alg{g}',\inner{\cdot}{\cdot})$. Let $c$ be a faithful nondegenerate representation of $C$ on some Hilbert space $\Hilbert_C$.

We fix some orthonormal basis $\{e_1,\ldots,e_n\}$ of $\alg{g'}$, and we let $X_1,\ldots,X_n \in \alg{g}$ be the dual basis. For each $j\in \{1,\ldots,n\}$, we define the derivation $\delta_j$ of $\A$ via $\alpha$ and $X_j$, by:
\begin{equation*}
\partial_j : a \longmapsto \lim_{h\rightarrow 0} \frac{\alpha_{\exp(h X_j)}(a) - a }{h}
\end{equation*}
wherever defined. Let $\A^1$ be the common domain of $\partial_1,\ldots,\partial_n$, which is a dense *-subalgebra in $\A$.

Let $\tau$ be the unique $\alpha$-invariant tracial state of $\A$. Let $\rho$ be the representation of $\A$ obtained from the Gel'fand-Naimark-Segal construction applied to $\tau$ and let $L^2(\A,\tau)$ be the corresponding Hilbert space. As $\A^1$ is dense in $L^2(\A,\tau)$, the operator $\partial_j$ defines an unbounded densely defined operator on $L^2(\A,\tau)$ for all $j \in \{1,\ldots,n\}$.

Let $\Hilbert = L^2(\A,\tau) \overline{\otimes} \Hilbert_C$ where $\overline{\otimes}$ is the standard tensor product for Hilbert spaces. We define the following representation of $\A$ on $\Hilbert$:
\begin{equation*}
\pi(a) : b \otimes f \longmapsto \rho(a)b \otimes f\text{.}
\end{equation*}

Let:
\begin{equation*}
H = \begin{pmatrix}
h_{11} & \cdots & h_{1n}\\
\vdots & & \vdots \\
h_{n1} & \cdots & h_{nn}
\end{pmatrix}
\end{equation*}
where for all $j, k \in \{1,\ldots,n\}$, the coefficients $h_{jk}$ are elements in the commutant of $\A$ in $L^2(\A,\tau)$, and where $H$ is invertible as an operator on $\Hilbert' = \oplus_{j=1}^n L^2(\A,\tau)$. We denote the identity over $\Hilbert'$ by $\unit_{\Hilbert'}$.

We define:
\begin{equation*}
D_H = \sum_{j=1}^n \sum_{k=1}^n h_{kj} \partial_k \otimes c(e_j)
\end{equation*}
so that for all $a\in\A^1$:
\begin{equation*}
[D_H,\pi(a)] = \sum_{j=1}^n \sum_{k=1}^n h_{kj} \rho(\partial_k(a)) \otimes c(e_j)\text{.}
\end{equation*}
We define, for all $a\in\A^1$:
\begin{equation*}
\mathsf{S}_H(a) = \opnorm{[D_H,\pi(a)]}{\Hilbert}\text{.}
\end{equation*}
Let $\Lip_H$ be the Minkoswky gauge functional of the convex set $\left\{ a \in \A^1 : \mathsf{S}_H(a)\leq 1 \right\}$.

Then:
\begin{enumerate}
\item $(\A,\Lip_H)$ is a {\Lqcms},
\item if we set:
\begin{equation*}
H' = \begin{pmatrix}
h_{11}' & \cdots & h_{1n}'\\
\vdots & & \vdots \\
h_{n1}' & \cdots & h_{nn}'
\end{pmatrix}
\end{equation*}
where $h_{jk}'$ lies in the commutant of $\rho(\A)$ for all $j,k \in \{1,\ldots,n\}$ and where $H'$ is invertible as an operator on $\Hilbert'$, then, defining $\Lip_{H'}$ similarly to $\Lip_H$, we conclude:
 \begin{multline*}
\qpropinquity{}((\A,\Lip_H),(\A,\Lip_{H'})) \leq  n \max\left\{ \opnorm{\unit_{\Hilbert'} - H' H^{-1}}{\Hilbert'}, \opnorm{\unit_{\Hilbert'}-H H'^{-1}}{\Hilbert'} \right\}\\ \times \left[1 + 
\frac{1}{2}\max\left\{\left( 1 + n\opnorm{1-H^{-1}}{\Hilbert'}\right)^{-1},\right.\right. \\ \left.\left. \left( 1 + n\opnorm{1-H'^{-1}}{\Hilbert'}\right)^{-1}\right\}\diam{\StateSpace(\A)}{\Kantorovich{\Lip}} \right]\text{.}
\end{multline*}
\end{enumerate}
\end{theorem}

We thus may apply Proposition (\ref{totally-bounded-prop}) when we restrict the parameter space $\mathcal{H}_C = \{H:\opnorm{H}{\B},\opnorm{H^{-1}}{\B}\leq C\}$ for any $C>0$. Of course this set itself is not totally bounded, yet the space $\{\Lip_H : H \in \mathcal{H}_C \}$ is totally bounded for both $\HausLip{\A}$ and $\qpropinquity{}$.
\end{example}

We conclude this section with another example of a compact class of quantum compact metric spaces obtained from perturbations. We include this example as another compact class of quantum metric spaces, derived from \cite[Lemma 3.79]{Latremoliere15}, though in this case it does not require Proposition (\ref{totally-bounded-prop}).

\begin{example}[Conformal Perturbations]\label{conformal-ex}

We proved in \cite{Latremoliere15b} that small conformal perturbations of quantum metrics are indeed close for the quantum propinquity. We recall the result here to fix our notations.

\begin{theorem}[{\cite[Theorem 3.81]{Latremoliere15b}}]\label{conformal-thm}
Let $\A$ be a unital C*-algebra, $\pi$ a faithful unital *-representation of $\A$ on some Hilbert space $\Hilbert$ and $D$ be a not necessarily bounded self-adjoint operator on $\Hilbert$ such that if $\Lip : a\in\sa{\A} \longmapsto \opnorm{[D,\pi(a)]}{\B}$ then $(\A,\Lip)$ is a {\Lqcms}. 

Let $\mathrm{GLip}(\A)$ be the set of all invertible elements $h$ in $\sa{\A}$ with $\Lip(h) < \infty$.  For any $h\in\mathrm{GLip}(\A)$, we define $D_h = \pi(h)D\pi(h)$, $\sigma_h : a\in\A\mapsto h^2 a h^{-2}$ and:
\begin{equation*}
\Lip_h : a\in \sa{\A} \longmapsto \opnorm{D_h\pi(a) - \pi(\sigma_h(a))D_h}{\B} \text{.}
\end{equation*}

Then $(\A,\Lip_h)$ is a {\Qqcms{\left(\|h^2\|_\A\|h^{-2}\|_\A, 0 \right)}} and moreover, if $(h_n)_{n\in\N}$ is a sequence in $\mathrm{GLip}(\A)$ which converges to $h \in \mathrm{GLip}$ and such that:
\begin{equation*}
\lim_{n\rightarrow\infty} \Lip(h_n^{-1}h) = \lim_{n\rightarrow\infty} \Lip(h_n h^{-1}) = 0\text{,} 
\end{equation*}
then:
\begin{equation*}
\lim_{n\rightarrow\infty} \qpropinquity{F_M}\left((\A,\Lip_{h_n}),(\A,\Lip_h)\right) = 0 \text{,}
\end{equation*}
where $M\geq \sup_{n\in\N} \|h_n^2\|_\A\|h_n^{-2}\|_\A$ and $F_M : x,y,u,v \in [0,\infty) \mapsto M(x v + y u)$.
\end{theorem}

Let $K_1, K_2, K_3 > 0$ and define:
\begin{equation*}
\Omega_{K_1,K_2,K_3} = \left\{ \omega \in \sa{\A} : \Lip(\omega)\leq K_1, \|\omega\|_{\A} \leq K_2, \|\omega^{-1}\|_\A\leq K_3 \right\}\text{.}
\end{equation*}
Then $\Omega_{K_1,K_2,K_3}$ is compact for $\|\cdot\|_\A$ since $\Lip$ is a lower semi-continuous Lip-norm. Theorem (\ref{conformal-thm}) shows that conformal perturbations are continuous for the quantum propinquity, and thus, using our notations, $\{\Lip_\omega : \omega \in \Omega_{K_1, K_2, K_3}\}$ is compact.
\end{example}

\section{Lipschitz distance between compact quantum metric spaces}

The Lipschitz distance between compact metric spaces \cite{Gromov} provides a distance between homeomorphic compact metric spaces based upon bi-Lipschitz isomorphisms, and thus it is natural to define it in this paper in light of our study of Lipschitz morphisms. 

This section provides the noncommutative generalization of the Lipschitz metric, which in essence is a metric on Lip-norms with common domains. The quantum Lipschitz distance is complete and dominates the quantum propinquity when working on appropriate classes of {\gQqcms s}. The Lipschitz distance also provides natural examples of totally bounded classes for the quantum propinquity, and thus compact classes for the dual propinquity \cite{Latremoliere13b}.

\begin{notation}
Let $(\A,\Lip_\A)$ and $(\B,\Lip_\B)$ be two quantum compact metric spaces and let $\varphi : \A\rightarrow\B$ a unital *-morphism. We denote by $\mathsf{dil}(\varphi)$ the Lipschitz seminorm of the dual map $\varphi : \mu \in (\StateSpace(\B), \Kantorovich{\Lip_\B}) \mapsto \mu\circ\varphi \in (\StateSpace(\A),\Kantorovich{\Lip_\A})$, i.e.:
\begin{equation*}
\mathsf{dil}(\varphi) = \sup\left\{ \frac{\Kantorovich{\Lip_\B}(\mu\circ\varphi,\nu\circ\varphi)}{\Kantorovich{\Lip_\A}(\mu,\nu)} : \mu,\nu \in \StateSpace(\B),\mu\not=\nu \right\}\text{,}
\end{equation*}
with the understanding that this quantity may be infinite. We refer to this quantity as the \emph{dilation factor}, or just \emph{dilation} of the given Lipschitz morphism.
\end{notation}

\begin{remark}
If $(\A,\Lip_\A)$ and $(\B,\Lip_\B)$ are two quantum compact metric spaces with lower semicontinuous Lip-norms, and if $\varphi : \A\rightarrow\B$ a unital *-morphism, then $\mathsf{dil}(\varphi) = \inf\left\{ C > 0 : \Lip_\B\circ\varphi \leq C \Lip_\A \right\}$ with the usual convention that $\inf\emptyset = \infty$.
\end{remark}

\begin{definition}
The \emph{Lipschitz distance} between two quantum compact metric spaces $(\A,\Lip_\A)$ and $(\B,\Lip_\B)$ is:
\begin{multline*}
\LipschitzD((\A,\Lip_\A),(\B,\Lip_\B)) =  \\
\inf\left\{ \max\left\{ \left|\ln(\mathsf{dil}(\varphi))\right|, \left| \ln(\mathsf{dil}(\varphi^{-1}))\right| \right\} : \varphi : \A\rightarrow\B \text{ is a *-isomorphism} \right\}\text{,}
\end{multline*}
with the conventions that $\inf\emptyset = \infty$ and $\ln(\infty) = \infty$.
\end{definition}

\begin{proposition}
If $(\A,\Lip_\A)$ and $(\B,\Lip_\B)$ are two quantum compact metric spaces with lower semicontinuous Lip-norms. Then:
\begin{multline*}
\LipschitzD((\A,\Lip_\A),(\B,\Lip_\B)) = \\
\inf\left\{\max\left\{ \left|\ln(\mathsf{dil}(\varphi))\right|, \left|\ln(\mathsf{dil}(\varphi^{-1}))\right|\right\} \middle\vert \begin{array}{l} 
\varphi : \A \rightarrow\B \text{ is a *-isomorphism} \\
\varphi(\dom{\Lip_\A}) = \dom{\Lip_\B}
\end{array}\right\} \text{,}
\end{multline*}
with the convention that $\inf\emptyset = \infty$.
\end{proposition}

\begin{proof}
This follows from Corollary (\ref{bi-Lipschitz-cor}).
\end{proof}

A natural class for the study of the Lipschitz distance is given by the following definitions, which includes all {\gQqcms s}: we simply require lower semi-continuity of the quantum metrics, as it fits the general framework of this paper, and we require that the domain of the Lip-norm is a Jordan-Lie algebra, to retain a minimum amount of information on the multiplicative structure of the C*-algebra.

\begin{definition}
A compact quantum metric space $(\A,\Lip)$ is \emph{Jordan-Lie} when $\Lip$ is lower semi-continuous, and its domain $\dom{\Lip}$ is a Jordan-Lie subalgebra of $\sa{\A}$.
\end{definition}

The Lipschitz distance between Jordan-Lie quantum compact metric spaces is actually achieved, as established in the following lemma. This observation will prove useful in establishing that the Lipschitz distance in indeed a distance up to quantum isometry.

\begin{lemma}\label{LipschitzD-reached-lemma}
If $(\A,\Lip_\A)$ and $(\B,\Lip_\B)$ are two Jordan-Lie compact quantum metric spaces such that $\LipschitzD((\A,\Lip_\A),(\B,\Lip_\B)) < \infty$ then there exists a *-isomorphism $\varphi : \A\rightarrow\B$ such that:
\begin{equation*}
\max\left\{ |\ln(\mathsf{dil}(\varphi))|, |\ln(\mathsf{dil}(\varphi^{-1}))| \right\} = \LipschitzD((\A,\Lip_\A),(\B,\Lip_\B)) \text{.}
\end{equation*}
\end{lemma}

\begin{proof}
Suppose that $\LipschitzD((\A,\Lip_\A),(\B,\Lip_\B)) = C$ for some $C \geq 0$. There exists a sequence of *-isomorphism $(\varphi_n)_{n\in\N}$ such that for all $n\in\N$ we have:
\begin{equation*}
C^{-1} \exp\left(-\frac{1}{n+1}\right)\Lip_\B\circ\varphi_n \leq \Lip_\A \leq C \exp\left(\frac{1}{n+1}\right)\Lip_\B\circ\varphi_n \text{.}
\end{equation*}

Let $a \in \sa{\A}$ with $\Lip_\A(a) < \infty$. Since $\|\varphi_n(a)\|_\B = \|a\|_\A$ and $\Lip_\B\circ\varphi_n(a) \leq 2C \Lip_\A(a)$ for all $n\in\N$, we conclude that $(\varphi_n(a))_{n\in\N}$ admits a convergent subsequence since $\Lip_\B$ is a Lip-norm. Let $\varphi_\infty(a)$ be its limit; as $\Lip_\B$ is lower semicontinuous, we conclude that $\Lip_\A(\varphi_\infty(a)) \leq 2 C \Lip_\A(a)$.

Since $\{ a \in \sa{\A} : \Lip_\A(a)\leq n, \|a\|\leq n \}$ is compact for the norm $\|\cdot\|_\A$, hence separable for all $n\in\N$, so is:
\begin{equation*}
\dom{\Lip_\A} = \left\{a\in\sa{\A} : \Lip_\A(a) < \infty  \right\} = \bigcup_{n\in\N} \left\{a\in\sa{\A} : \Lip_\A(a)\leq n, \|a\|_\A\leq n \right\}\text{.}
\end{equation*}

Let $\alg{F}$ be a countable dense subset of $\{a\in\sa{\A}:\Lip_\A(a) < \infty \}$. A diagonal argument proves that there exists a subsequence $(\varphi_{f(n)})_{n\in\N}$ such that for all $a\in\alg{F}$ we have $(\varphi_{f(n)}(a))_{n\in\N}$ converges uniformly to $\varphi_\infty(a)$ (see \cite[Theorem 5.13]{Latremoliere13}).

Moreover, if $a\in\sa{\A}$ with $\Lip_\A(a)<\infty$, then for all $\varepsilon>0$, there exists $a_\varepsilon \in \alg{F}$ with $\|a-a_\varepsilon\|<\frac{\varepsilon}{3}$. Let $N\in\N$ be such that for all $p,q \geq N$, we have $\|\varphi_{f(p)}(a_\varepsilon) - \varphi_{f(q)}(a_\varepsilon)\|_\B \leq \frac{\varepsilon}{3}$. Thus for all $p,q\geq N$, we have:
\begin{equation*}
\begin{split}
|\varphi_{f(p)}(a) - \varphi_{f(q)}(a)| &\leq |\varphi_{f(p)}(a-a_\varepsilon)| + |\varphi_{f(p)}(a_\varepsilon) - \varphi_{f(q)}(a_\varepsilon)| + |\varphi_{f(q)}(a-a_\varepsilon)| \\
&\leq \frac{\varepsilon}{3} + \frac{\varepsilon}{3} + \frac{\varepsilon}{3} = \varepsilon\text{.}
\end{split}
\end{equation*}
Thus $(\varphi_{f(n)}(a))_{n\in\N}$ converges as well, since it is a Cauchy sequence in $\A$ which is complete. Its limit is denoted once more by $\varphi_\infty(a)$. 

Note that since for all $n\in\N$ and for all $a\in\dom{\Lip_\A}$, we have $\|\varphi_n(a)\|_\B = \|a\|_\A$, we also have $\|\varphi_\infty(a)\|_\B = \|a\|_\A$. We thus have defined an isometric map $\varphi_\infty : \dom{\Lip_\A} \rightarrow\sa{\B}$. Moreover, as a pointwise limit of Jordan-Lie morphisms, $\varphi_\infty$ is also a Jordan-Lie morphisms on $\dom{\Lip}$.

Now $\Lip_\B$ is lower semi-continuous and, for all $n\in\N$ we have $\Lip_\B\circ\varphi_n(a) \leq C \exp(\frac{1}{n+1}) \Lip_\A(a)$. Thus $\Lip_\B\circ\varphi_\infty(a) \leq C \Lip_\A(a)$. Thus $\mathsf{dil}(\varphi_\infty) \leq C$.

Thus $\varphi_\infty$ extends by continuity to a Jordan-Lie morphism from $\sa{\A}$ to $\sa{\B}$. Our argument is now concluded in the same manner as \cite[Claim 5.18, Theorem 5.13]{Latremoliere13} and proves that $\varphi_\infty$ extends to a unital *-morphism from $\A$ to $\B$ with $\mathrm{dil}(\varphi_\infty) \leq C$.

The same method may be applied to construct some subsequence of $\left(\varphi_n^{-1}\right)_{n\in\N}$ converging pointwise on $\dom{\Lip_\B}$ to some *-morphism $\psi_\infty$ on $\B$ with $\Lip_\A\circ\psi_\infty \leq C \Lip_\A$. Up to extracting further subsequences, we shall henceforth assume that both $(\varphi_{f(n)})_{n\in\N}$ and $(\varphi_{f(n)}^{-1})_{n\in\N}$ converge pointwise to, respectively $\varphi_\infty$ on $\dom{\Lip_\A}$ and $\psi_\infty$ on $\dom{\Lip_\B}$. It is then immediate to check that $\varphi_\infty\circ\psi_\infty$ is the identity of $\dom{\Lip_\B}$ and $\psi_\infty\circ\varphi_\infty$ is the identity on $\dom{\Lip_\A}$. Then by construction, $\psi_\infty\circ\varphi_\infty$ is the identity on $\A$ and $\varphi_\infty\circ\psi_\infty$ is the identity on $\B$. Thus $\varphi_\infty$ is a *-isomorphism from $\A$ to $\B$.

In particular, we also obtain that $\Lip_\A\circ \varphi_\infty^{-1} \leq C \Lip_\B$ and thus $\mathsf{dil}(\varphi^{-1})\leq C$. As we may not have both $\mathsf{dil}(\varphi) < C$ and $\mathsf{dil}(\varphi^{-1}) < C$, since $C$ is the infimum of the dilations of such *-isomorphisms, the lemma is proven.
\end{proof}

We now establish that the Lipschitz distance is indeed, a distance up to quantum isometry, and that it dominates the quantum propinquity. 

\begin{theorem}\label{LipschitzD-distance-thm}
The Lipschitz distance is an extended metric up to quantum isometry on the class of Jordan-Lie quantum compact metric spaces. Explicitly, for all $(\A,\Lip_\A)$, $(\B,\Lip_\B)$ and $(\D,\Lip_\D)$ Jordan-Lie compact quantum metric spaces, we have:
\begin{enumerate}
\item $\LipschitzD((\A,\Lip_\A),(\B,\Lip_\B)) \in [0,\infty]$, and is finite if and only if there exists a *-isomorphism $\varphi : \A \rightarrow \B$ such that $\varphi(\dom{\Lip_\A}) = \dom{\Lip_\B}$,
\item $\LipschitzD((\A,\Lip_\A),(\D,\Lip_\D)) \leq \LipschitzD((\A,\Lip_\A),(\B,\Lip_\B)) + \LipschitzD((\B,\Lip_\B),(\D,\Lip_\B))$,
\item $\LipschitzD((\A,\Lip_\A),(\B,\Lip_\B)) = \LipschitzD((\B,\Lip_\B),(\A,\Lip_\A))$,
\item $\LipschitzD((\A,\Lip_\A),(\B,\Lip_\B)) = 0$ if and only if $(\A,\Lip_\A)$ and $(\B,\Lip_\B)$ are isometrically isomorphic, i.e. there exists a *-isomorphisms $\varphi : \A\rightarrow\B$ such that $\Lip_\B\circ\varphi = \Lip_\A$ on $\sa{\A}$. Such a map $\varphi$ is called a \emph{quantum isometry} or an \emph{isometric isomorphism}.
\item if $(\A,\Lip_\A)$ and $(\B,\Lip_\B)$ are two {\Qqcms{F}s} then:
\begin{multline*}
\qpropinquity{F}((\A,\Lip_\A),(\B,\Lip_\B)) \leq 
\left(\exp\left(\LipschitzD((\A,\Lip_\A),(\B,\Lip_\B))\right)-1\right) \\
\cdot\left(1 + \max\left\{\diam{\StateSpace(\A)}{\Kantorovich{\Lip_\A}}, \diam{\StateSpace(\B)}{\Kantorovich{\Lip_\B}}\right\}\right) \text{.}
\end{multline*}
\end{enumerate}
\end{theorem}

\begin{proof}
The function $\LipschitzD$ is valued in $[0,\infty]$ by definition, and finite if and only if there exists a bi-Lipschitz isomorphism between its two  arguments, by Corollary (\ref{bi-Lipschitz-cor}). It is symmetric in its two arguments by construction. The triangle inequality follows from simple computations as well.

The last assertion of our proposition follows immediately from \cite[Proposition 3.80]{Latremoliere15b}.

If there exists an isometric isometry between two compact quantum metric spaces, then their Lipschitz distance is null. Only the converse of this observation requires our assumption that the domain of Lip-norms be Jordan-Lie algebras. We simply apply Lemma (\ref{LipschitzD-reached-lemma}).
\end{proof}

We now prove that closed balls for the Lipschitz distance are compact classes for the dual propinquity.

\begin{theorem}\label{LipschitzD-compact-thm}
Let $(\A,\Lip_\A)$ be a {\Qqcms{F}} for some admissible function $F$. If $R \geq 0$ then the class $\mathcal{B}$ of {\Qqcms{F}s} in the closed ball of center $(\A,\Lip_\A)$ and radius $R$ for the Lipschitz distance $\LipschitzD$ is totally bounded for the quantum Gromov-Hausdorff propinquity.

Therefore, the closure of $\mathcal{B}$ for the dual propinquity is compact.
\end{theorem}

\begin{proof}
For all $(\B,\Lip)$ within Lipschitz distance $R$ of $(\A,\Lip_\A)$, there exists by definition a *-isomorphism $\varphi_\B : \A \rightarrow \B$ which maps the domain of $\Lip_\A$ into the domain of $\Lip$. We set $\Lip' = \Lip\circ\varphi_\B$ and note that $\Lip'$ is a lower semicontinuous Lip-norm with the same domain as $\Lip_\A$. Moreover, by definition of the Lipschitz distance, we have $\Lip_\A \leq \exp(R) \Lip'$. Last, $(\A,\Lip')$ and $(\B,\Lip)$ are isometrically isomorphic, thus their propinquity is zero, and so is their Lipschitz distance.

Thus by Proposition (\ref{totally-bounded-prop}), the closed ball $\mathcal{B}$ of center $(\A,\Lip_\A)$ and radius $R$ is totally bounded for $\HausLip{\A}$. By Proposition (\ref{HausLip-qprop-prop}), the subclass of {\Qqcms{F}s} in the closed ball $\mathcal{B}$ is totally bounded for the quantum propinquity. The rest of the theorem follows from the completeness of the dual propinquity \cite{Latremoliere13b} and the dominance of the quantum propinquity over the dual propinquity.
\end{proof}

We conclude this section by proving that the Lipschitz distance is indeed complete.

\begin{theorem}
The distance $\LipschitzD$ is complete on the class of Jordan-Lie quantum compact metric spaces.
\end{theorem}

\begin{proof}
As in the proof of Theorem (\ref{LipschitzD-compact-thm}), we can assume that we are given a sequence $(\Lip_n)_{n\in\N}$ of lower semi-continuous Lip-norms on some unital C*-algebra, such that $(\A,\Lip_n)_{n\in\N}$ is a Cauchy sequence for $\LipschitzD$. 

By Theorem (\ref{LipschitzD-distance-thm}), the sequence $\alg{L}_n = \{ a \in \sa{\A} : \Lip_n(a) \leq 1 \}$ is Cauchy for $\HausLip{\|\cdot\|_\A}$. As the latter metric is complete, $(\alg{L}_n)_{n\in\N}$ converges to some $\alg{L}$. It is easy to check that $\alg{L}$ thus defined is a convex closed subset of $\sa{\A}$. Let $\mathsf{L}$ be its Minkowsky gauge functional.

Let $\varepsilon \in (0,1)$. There exists $N\in\N$ such that for all $n\geq N$ we have:
\begin{equation*}
(1-\varepsilon)\Lip_q \leq \Lip_p \leq (1+\varepsilon) \Lip_q
\end{equation*}
 for all $p,q \geq N$. In other words, $\alg{L}_q \subseteq (1+\varepsilon) \alg{L}_p$ and $\alg{L}_p \subseteq \frac{1}{1-\varepsilon}\alg{L}_q$ for all $q, p \geq N$. Now $(1+\varepsilon)\alg{L}_n$ is closed for all $n\in\N$, thus the hyperspace of its closed subsets is complete and thus closed for the Hausdorff distance $\Haus{\|\cdot\|_\A}$. Consequently, $\alg{L} \subseteq (1+\varepsilon) \alg{L}_n$ and $\frac{1}{1+\varepsilon}\alg{L}_n\subseteq\alg{L}$ for all $n\geq N$. Thus, $\dom{\Lip} = \dom{\Lip_n}$ for $n\geq N$, and thus $\dom{\Lip}$ is a dense Jordan-Lie subalgebra of $\sa{\A}$. Moreover, we conclude that $\Lip_n \leq (1+\varepsilon)\Lip$. In particular, by \cite[Lemma 1.10]{Rieffel98a}, $(\A,\Lip)$ is a {\qcms}.

% On the other hand, let $a\in\alg{L}_n$ for $n\geq N$. For $p\geq N$ we have $a\in\frac{1}{1-\varepsilon}\alg{L}_p$ so $a \in \frac{1}{1-\varepsilon}\alg{L}$ since $ \frac{1}{1-\varepsilon}\alg{L}$ is the Hausdorff limit of $s = \left( \frac{1}{1-\varepsilon}\alg{L}\right)_{p \geq n}$ and thus contains all the limits of convergent sequences obtained by picking one element in each set of $s$.

We also conclude from these inclusions that $\LipschitzD((\A,\Lip),(\A,\Lip_n)) \leq \ln(1+\varepsilon)$ for $n\geq N$. Our theorem is now proven.
\end{proof}

\section{A metric for pointwise convergence on the automorphism group of a quantum compact metric space}

We now introduce a new metric on the automorphism group of a quantum compact metric space. Our motivation is given by Theorem (\ref{equiv-thm}), and in particular Corollary (\ref{auto-Lipschitz-cor}), as well as our new understanding of compactness for Lip-norms with a shared domain.

\newcommand{\KantorovichLength}[1]{{\mathrm{mk}\ell_{#1}}}
\begin{definition}
Let $(\A,\Lip)$ be a quantum compact metric space. For any *-automorphism of $\A$, we define:
\begin{equation*}
\KantorovichLength{\Lip}(\alpha) = \sup\left\{ \|\alpha(a) - a\|_\A : a\in\dom{\Lip}, \Lip(a)\leq 1\right\}\text{.}
\end{equation*}
\end{definition}

\begin{proposition}
Let $(\A,\Lip)$ be a quantum compact metric space.
\begin{enumerate}
\item $\KantorovichLength{\Lip}$ is a length function on $\mathrm{Aut}(\A)$ which metrizes the topology of pointwise convergence,
\item $\diam{\mathrm{Aut}(\A)}{\KantorovichLength{\Lip}} \leq \diam{\StateSpace(\A)}{\Kantorovich{\Lip}}$
\item If $\Lip$ is closed, $\Xi$ is a subset of $\mathrm{Aut}(\A)$ and there exists $C > 0$ such that:
\begin{equation*}
\forall \alpha \in \Xi \quad \Lip\leq C \Lip\circ\alpha
\end{equation*}
then $\Xi$ is totally bounded for $\KantorovichLength{\Lip}$.
\end{enumerate}
\end{proposition}

\begin{proof}
\begin{claim}
$\KantorovichLength{\Lip}$ is a length function on the group $\mathrm{Aut}(\A)$.
\end{claim}

if $\mathrm{id}_\A$ is the identity of $\A$ then $\KantorovichLength{\Lip}(\mathrm{id}_\A) = 0$. Now let $\alpha\in\mathrm{Aut}(\A)$.

If $\KantorovichLength{\Lip}(\alpha) = 0$ then $\|a-\alpha(a)\|_\A = 0$ for all $a\in\dom{\Lip}$ with $\Lip(a)\leq 1$. Thus $\|a-\alpha(a)\|_\A = 0$ for all $a\in\dom{\Lip}$, and by continuity of $\alpha$ and density of $\dom{\Lip}$, we conclude that $\|a-\alpha(a)\|_\A = 0$ for all $a\in\sa{\A}$. Thus by linearity, $\alpha(a) = a$ for all $a\in\A$.

Moreover, since $\alpha$ is an isometry of $(\A,\|\cdot\|_\A)$, we have for all $a\in\A$:
\begin{equation*}
\|a-\alpha(a)\|_\A = \|\alpha(\alpha^{-1}(a) - a)\|_\A = \|\alpha^{-1}(a)-a\|_\A
\end{equation*}
from which it follows that $\KantorovichLength{\Lip}(\alpha) = \Kantorovich{\Lip}\left(\alpha^{-1}\right)$.

Let now $\beta\in\mathrm{Aut}(\A)$. For all $a\in\A$:
\begin{equation*}
\begin{split}
\|\alpha\circ\beta(a) - a\|_\A &\leq \|\alpha\circ\beta(a) - \alpha(a)\|_\A + \|\alpha(a) - a\|_\A\\
&\leq \|\alpha(\beta(a) - a)\|_\A + \|\alpha(a) - a\|_\A\\
&\leq \|\beta(a) - a\|_\A + \|\alpha(a) - a\|_\A\text{,}
\end{split}
\end{equation*}
from which we conclude:
\begin{equation*}
\KantorovichLength{\Lip}(\alpha\circ\beta) \leq \KantorovichLength{\Lip}(\beta) + \KantorovichLength{\Lip}(\alpha)\text{.}
\end{equation*}

\begin{claim}
The topology induced by $\KantorovichLength{\Lip}$ is the topology of pointwise convergence.
\end{claim}
Let $(\alpha_j)_{j\in J}$ be a net in $\mathrm{Aut}(\A)$ which converges pointwise to some $\alpha \in \mathrm{Aut}(\A)$ over $\A$, where $(J,\succ)$ is a directed set. Let $\varepsilon > 0$.

Let $\varphi \in \StateSpace(\A)$. Since $\alg{B} = \{a\in\dom{\Lip} : \Lip(a)\leq 1,\varphi(a) = 0\}$ is totally bounded, there exists $\alg{F}\subseteq\alg{B}$ with $\alg{F}$ finite and such that for all $a\in\alg{B}$ there exists $f(a) \in \alg{F}$ with $\|a-f(a)\|_\A \leq \frac{\varepsilon}{3}$.

For each $a\in\alg{F}$, there exists $j_a \in J$ such that if $j\succ j_a$, we have $\|\alpha_j(a)-\alpha(a)\|_\A\leq \frac{\varepsilon}{3}$. As $J$ is directed and $\alg{F}$ is finite, there exists $j'\in J$ with $j' \succ j_a$ for all $a\in\alg{F}$. Thus, if $a\in\alg{B}$ and $j\succ j'$, then:
\begin{multline*}
\|\alpha_j(a)-\alpha(a)\|_\A \leq \|\alpha_j(a) -\alpha_j(f(a))\|_\A + \|\alpha_j(f(a)) - \alpha(f(a))\|_\A \\ + \|\alpha(f(a)) - \alpha(a)\|_\A
\leq \varepsilon\text{,}
\end{multline*}
where we used that all automorphisms are isometries. Thus, for all $j\succ j'$ we have:
\begin{equation*}
\KantorovichLength{\Lip}(\alpha_j^{-1}\circ\alpha) \leq\varepsilon\text{.}
\end{equation*}

Conversely, let $(\alpha_j)_{j\in J}$ be a net in $\mathrm{Aut}(\A)$ converging for $\KantorovichLength{\Lip}$ to $\alpha\in\mathrm{Aut}(\A)$. Let $a\in\sa{\A}$ and $\varepsilon > 0$. 

Since $\dom{\Lip}$ is dense, there exists $b\in\dom{\Lip}$ such that $\|a-b\|_\A \leq \frac{\varepsilon}{3}$. Now, by assumption, there exists $j_0 \in J$ such that for all $j\succ j_0$, we have $\KantorovichLength{\Lip}(\alpha_j^{-1}\circ\alpha)\leq\frac{\varepsilon}{3(\Lip(b)+1)}$. Thus $\|\alpha_j(b)-\alpha(b)\|_\A \leq \frac{\varepsilon}{3}$ and thus for all $j\succ j_0$:
\begin{equation*}
\|\alpha_j(a) - \alpha(a)\|_\A \leq \|\alpha_j(a)-\alpha_j(b)\|_\A + \|\alpha_j(b)-\alpha(a)\|_\A + \|\alpha(a) - \alpha_j(a)\|_\A \leq \varepsilon\text{,}
\end{equation*}
as desired.

\begin{claim}
Assume $\Lip$ is closed. The set $\{\alpha\in\mathrm{Aut}(\A) : \Lip\circ\alpha \leq C\Lip \}$, for some fixed $C > 0$, is totally bounded and closed for $\KantorovichLength{\Lip}$.
\end{claim}

Fix $\varphi \in \StateSpace(\A)$.

All automorphisms of $\A$ are isometries and thus $\mathrm{Aut}(\A)$ is equicontinuous over the compact set $\alg{X} = \{a\in\dom{\Lip}:\Lip(a)\leq 1, \varphi(a) = 0 \}$. Moreover, for all $a\in\dom{\Lip}$ with $\Lip(a)\leq 1$, then $\Lip\circ\alpha(a)\leq C$. Thus, $\alpha(a) \in \{a\in\dom{\Lip}:\Lip(a)\leq C, \varphi(a)=0\}$ for all $a\in\dom{\Lip}$ with $\Lip(a)\leq 1$ with $\varphi\circ\alpha^{-1}(a) = 0$.

Thus, by Arz{\'e}la-Ascoli, the set $\{\alpha\in\mathrm{Aut}(\A):\Lip\circ\alpha\leq C\Lip\}$ is totally bounded for the supremum norm over $\alg{X}$. 

We note that as automorphisms are unital, the supremum over $\alg{X}$ of the difference of two automorphisms of $\A$ equals to the supremum over $\{a\in\sa{\A} : \Lip(a)\leq 1\}$. Therefore, $\{\alpha\in\mathrm{Aut}(\A):\Lip\circ\alpha\leq C\Lip\}$ is totally bounded for $\KantorovichLength{\Lip}$.

Moreover, if $\Lip\circ\alpha_n\leq C\Lip$ for all $n\in\N$, by lower semicontinuity, $\Lip\circ\alpha\leq C\Lip$ as well. Thus our set is closed. As $\KantorovichLength{\Lip}$ is complete, our proof is complete.
\end{proof}

\providecommand{\bysame}{\leavevmode\hbox to3em{\hrulefill}\thinspace}
\providecommand{\MR}{\relax\ifhmode\unskip\space\fi MR }
% \MRhref is called by the amsart/book/proc definition of \MR.
\providecommand{\MRhref}[2]{%
  \href{http://www.ams.org/mathscinet-getitem?mr=#1}{#2}
}
\providecommand{\href}[2]{#2}

\vfill

\end{document}